\documentclass[12pt, a4paper, headsepline]{scrartcl}

\usepackage[english]{babel}
\usepackage{amssymb}
\usepackage{amsmath}
\usepackage{amsthm}
\usepackage{amsfonts}
\usepackage{graphics}
\usepackage{amssymb,psfrag,lscape,epsf,color,listings}
\usepackage{scrpage2}
\usepackage{authblk}

\newcommand{\Mod}[1]{\ (\text{mod}\ #1)}

\usepackage{ifthen}       
\usepackage{scrtime}      

\ifthenelse{\isundefined \draft }
{
  \newcommand{\look}[1]{}
  \newcommand{\lookM}[1]{}%

  \automark[subsection]{section}  %
}
{
  
  \newcommand{\markerM}{\fbox{\rule{0pt}{0.1ex}\textbf{Michela}}}
  \newcommand{\look}[1]{\textbf{*}
    \footnote{ #1 }}
  \newcommand{\lookM}[1]{\markerM\textbf{*}
    \footnote{\textbf{Michela:} #1 }}

  \rehead{\parbox[t]{1.0\textwidth}%
    {\textsf{\texttt{\jobname.tex} \quad --- DRAFT --- \quad\\
        \today, \thistime}}}
  \lohead{\parbox[t]{1.0\textwidth}%
    {\textsf{\texttt{\jobname.tex} \quad --- DRAFT ---  \quad\\
        \today, \thistime}}}
}

\newtheoremstyle{mythmstyle}
  {\topsep}
  {\topsep}
  {\itshape}
  {}
  {\bfseries \sffamily}
  {.}
  {.5em}
  {}

\newtheoremstyle{mydefstyle}
  {\topsep}
  {\topsep}
  {\normalfont}
  {}
  {\bfseries \sffamily}
  {.}
  {.5em}
  {}

\theoremstyle{mythmstyle}
\newtheorem{thm}{Theorem}[section]
\newtheorem{prop}[thm]{Proposition}
\newtheorem{cor}[thm]{Corollary}

\newtheorem{lemma}[thm]{Lemma}
\theoremstyle{mydefstyle}
\newtheorem{defin}[thm]{Definition}
\newtheorem{ex}[thm]{Example}

\newtheorem{rmk}[thm]{Remark}

\newcommand{\R}{\mathbb{R}} 
\newcommand{\C}{\mathbb{C}} 
\newcommand{\N}{\mathbb{N}} 
\newcommand{\Z}{\mathbb{Z}} 

\DeclareMathOperator{\grad}{grad} 
\DeclareMathOperator{\dvg}{div} 	
\DeclareMathOperator{\hess}{Hess}	
\DeclareMathOperator{\ric}{Ric}		
\DeclareMathOperator{\tr}{tr}			

\newcommand{\conj}[1]{\overline{#1}} 

\DeclareMathOperator{\gradA}{\grad^{\alpha}}		
\DeclareMathOperator{\dvgA}{\dvg^{\alpha}}			
\DeclareMathOperator{\hessA}{Hess^{\alpha}}			
\newcommand{\nablaA}[2]{\nabla^{\alpha}_{#1}#2}	

\DeclareMathOperator{\dvol}{dvol} 

\newcommand{\normsymb}{\|}
\newcommand{\normsqr}[1]{\normsymb #1 \normsymb^2}

\newcommand{\abs}{\vert}
\newcommand{\abssqr}[1]{\abs #1 \abs^2}

\title{Ricci curvature and eigenvalue estimates for the magnetic Laplacian on manifolds}
  \ifthenelse{\isundefined \draft}{}
  {{\newline \upshape{--- DRAFT---}}}

\author{Michela Egidi, Shiping Liu, Florentin M\"unch, and
Norbert Peyerimhoff}

\ifthenelse{\isundefined \draft}
{\date{\today}}  
{\date{\today, \thistime,  \emph{File:} \texttt{\jobname.tex}}}

\begin{document}

\maketitle

\begin{abstract}
In this paper, we present a Lichnerowicz type estimate and (higher order) Buser type estimates for the magnetic Laplacian on a closed Riemannian manifold with a magnetic potential. These results relate eigenvalues, magnetic fields, Ricci curvature, and Cheeger type constants.
\end{abstract}

\section{Introduction}
In recent decades, the spectral theory of the magnetic Laplacian has attracted a lot of attention
on various spaces: on domains (see, e.g., \cite{erdos:96, flm:08}), on noncompact Riemannian manifolds (see, e.g., \cite{mt:12, shubin:01}), on discrete graphs (see, .e.g, \cite{shubin:94, sunada:93}), and on fractals (see, e.g., \cite{ht:13}), just to name a few. In this article, we are particularly interested in eigenvalue estimates of the magnetic Laplacian on a closed (i.e. compact without boundary) connected Riemannian manifold. In contrast to the Laplace-Beltrami operator (Laplacian, for short), whose smallest eigenvalue is equal to zero and simple, the smallest eigenvalue of the magnetic Laplacian can be positive and of higher multiplicity. Most of the existing eigenvalue estimates are concerned with the smallest eigenvalue. Shigekawa \cite{shigekawa:87} proved a comparison result for the smallest eigenvalue and studied the asymptotic behaviour of eigenvalues of the magnetic Laplacian. Paternain \cite{paternain:01} obtained an upper bound of the smallest eigenvalue in terms of a harmonic value and a critical value of the corresponding Lagrangian. The magnetic Laplacian fits into the more general framework of the  connection Laplacian. Ballmann, Br\"uning and Carron \cite{bbc:03} proved lower bound estimates of the smallest eigenvalue of the connection Laplacian for Hermitian vector bundles over a closed Riemannian manifold in terms of a holonomy constant. Recently, Cheeger type estimates for all eigenvalues of the magnetic Laplacian were established in \cite{llpp:15}.

In this article, we are interested in the interaction between eigenvalues of the magnetic Laplacian, Ricci curvature of the underlying closed Riemannian manifold, magnetic Cheeger constants, and the magnetic field. Building upon a Bochner formula involving both the classical Laplacian and the magnetic Laplacian (see Theorem \ref{thm:magnetic_bochner} below) and inspired by earlier investigations of the discrete counterpart in \cite{lmp:15}, we obtain two eigenvalue estimates: the first result provides information about the first two eigenvalues and, in particular, a spectral gap in the case of positive Ricci curvature and small magnetic fields, and the other one is concerned with all eigenvalues of the magnetic Laplacian in terms of a non-positive lower Ricci curvature bound and the magnetic Cheeger constants. These estimates extend the classical Lichnerowicz and Buser estimates for eigenvalues of the Laplacian and provide new insights.

Let us now fix some notation. Let $(M, g)$ be a closed $n$-dimensional Riemannian manifold. Throughout this paper, we assume that $M$ is connected. Let $\alpha$ be a smooth real differential $1$-form on $M$, which is called the \emph{magnetic potential}. Let $0\leq \lambda_1^\alpha\leq \lambda_2^\alpha\leq \cdots \nearrow \infty$ be the eigenvalues of the magnetic Laplacian $\Delta^\alpha$ and $0= \lambda_1< \lambda_2\leq \cdots \nearrow \infty$ be the eigenvalues of the classical Laplacian, both ordered increasingly and counted with multiplicity.

The classical Lichnerowicz estimate states that $\lambda_2\geq (n/(n-1))K$, whenever $K > 0$ is a lower bound of the Ricci curvature of the manifold $M$. Recall that $\lambda_1=0$ and that $\lambda_1$ is simple. In other words, the Lichnerowicz estimate establishes a spectral gap between $\lambda_1=0$ and $\lambda_2 > 0$ of size at least $(n/(n-1))K$.  As already mentioned above, the smallest eigenvalue $\lambda_1^\alpha$ of the magnetic Laplacian $\Delta^\alpha$ can be positive and can even be equal to $\lambda_2^\alpha$. Using the Bochner type formula in Theorem \ref{thm:magnetic_bochner}, we show in the case of positive Ricci curvature that there is also an interval of positive length between the first two eigenvalues $\lambda_1^\alpha$ and $\lambda_2^\alpha$ of the magnetic Laplacian, provided the \emph{magnetic field} $d\alpha$ is not too large. More explicitly, we have the following result.

\begin{thm}[Magnetic Lichnerowicz Theorem]\label{thm:Lichnerowicz}
Let $M$ be a closed Riemannian manifold of dimension $n\geq 2$ with a magnetic potential $\alpha \in \Omega^1(M)$. If
$$\ric\geq K>0\qquad \text{and}\qquad\lVert d\alpha \rVert_\infty\leq \left(1+2\sqrt{\frac{n-1}{n}}\right)^{-1}K,$$
then we have
	\begin{equation}\label{eq:Lich}
	0\leq\lambda^\alpha_1\leq a_-(K,\lVert d\alpha\rVert_\infty,n)\qquad\text{ and }\qquad\lambda^\alpha_2\geq a_+(K,\lVert d\alpha\rVert_\infty,n),
	\end{equation}
where
	\begin{equation*}
	a_{\pm}(K,\lVert d\alpha\rVert_\infty,n)=n\cdot\frac{(K-\lVert d\alpha\rVert_\infty)\pm \sqrt{(K-\lVert d\alpha\rVert_\infty)^2-4(\frac{n-1}{n})\normsqr{d \alpha}_\infty}}{2(n-1)}.
	\end{equation*}
	
In particular, there is a spectral gap
	\begin{equation}\label{eq:Lich_gap}
	\lambda^\alpha_2-\lambda^\alpha_1\geq \frac{n}{n-1}\sqrt{(K-\lVert d\alpha\rVert_\infty)^2-\frac{4(n-1)}{n}\normsqr{d \alpha}_\infty}.
	\end{equation}
\end{thm}
Note that when the magnetic potential vanishes or, more generally,
when $\alpha$ can be gauged away (for example, when $\alpha$ is exact), the
above result reduces to the classical Lichnerowicz estimate.

Cheeger's isoperimetric constant provides an important geometric lower bound for $\lambda_2$ of $\Delta$, which is well-known as Cheeger's inequality \cite{cheeger:70}. Later, Buser \cite{buser:82} showed an upper bound of $\lambda_2$ in terms of Cheeger's constant, with a constant depending on the dimension and the Ricci curvature of $M$. Ledoux \cite{ledoux:04} established a dimension-free Buser inequality.
Cheeger inequalities were extended to the magnetic Laplacian on a closed Riemannian manifold in \cite{llpp:15}. In particular, a \emph{$k$-way Cheeger type constant} $h_k^\alpha$ was introduced for the magnetic Laplacian $\Delta^\alpha$. The constant $h_k^\alpha$ is based on a mixture of the classical isoperimetric area/volume ratios of domains $\Omega \subset M$ and the \emph{frustration index} of the magnetic potential $\alpha$. The frustration index measures, in some sense, the non-triviality of $\alpha$ over $\Omega$. In particular, the frustration index vanishes if and only if $\alpha$ can be gauged away on $\Omega$. The readers can find the precise definitions in
Subsection \ref{section:gauge_Cheeger}. It was shown in \cite{llpp:15} that $h_k^\alpha\leq Ck^3\sqrt{\lambda_k^\alpha}$ for some absolute dimension-independent constant $C>0$.

Building upon the Bochner type formula in Theorem
\ref{thm:magnetic_bochner} and techniques developed in Ledoux
\cite{ledoux:04} and in \cite{lmp:15}, we prove the following upper
estimate for $\lambda_k^\alpha$ in the case of $d\alpha=0$ in terms of
$h_k^\alpha$ and a time parameter restricted by the the lower Riccci
curvature bound $-K$. Note that a potential $\alpha$ satisfying
$d\alpha = 0$ can still be non-trivial (see Example \ref{ex:S1three}
and Remark \ref{rm:dal} for an explanation).

\begin{thm}\label{thm:florentin}
Let $(M,g)$ be a closed Riemannian manifold with a magnetic potential $\alpha$ such that $d\alpha=0$. Let $-K$, $K\geq 0$ be a lower bound of the Ricci curvature of $M$. Then for any $0\leq t\leq \frac{1}{2K}$ and any $k\in \N$, we have
\begin{equation}
 2\sqrt{t}\cdot h_k^\alpha\geq \frac{1}{k}-e^{-t\lambda_k^\alpha}.
\end{equation}
\end{thm}

In the special case $K=0$ in Theorem \ref{thm:florentin}, non-trivial
magnetic potentials (i.e., potentials which cannot be gauged away)
exist only in the case of \emph{Ricci-flat manifolds}, due to the
classical Bochner vanishing theorem for the first cohomology (see,
e.g., \cite[Thm. 3.5.1]{jost:08} and the remark thereafter). Note,
however, in contrast to Theorem \ref {thm:Lichnerowicz}, that Theorem
\ref{thm:florentin} holds also for Riemannian manifolds with negative
Ricci curvature.

Theorem \ref{thm:florentin} can be considered as a higher order Buser inequality for every order $k \in \N$. In particular, when $k=1$, it implies immediately the following (dimension-free) Buser type inequality.

\goodbreak

\begin{thm}[Magnetic Buser inequality]\label{thm:Buser}
Let $(M,g)$ be a closed Riemannian manifold with a magnetic potential $\alpha$ such that $d\alpha=0$. Let $-K$, $K\geq 0$ be a lower bound of the Ricci curvature of $M$. Then we have
\begin{equation}
\lambda^\alpha_1\leq \max\left\{4\sqrt{2K}h_1^\alpha, \frac{4e^2}{(e-1)^2}(h_1^\alpha)^2\right\}.
\end{equation}
\end{thm}

\begin{proof}[Proof of Theorem \ref{thm:Buser}] Theorem \ref{thm:florentin} states that, for any $0\leq t\leq \frac{1}{2K}$,
\begin{align*}
  2\sqrt{t}\cdot h_1^\alpha \geq 1-e^{-t\lambda_1^\alpha}.
\end{align*}
When $\lambda_1^\alpha>0$ and $\lambda_1^\alpha\geq 2K$, we choose $t=\frac{1}{\lambda_1^\alpha}$, and obtain $\sqrt{\lambda_1^\alpha}\leq \frac{2e}{e-1}h_1^\alpha$. If, otherwise, $\lambda_1^\alpha\leq 2K$, we choose $t=\frac{1}{2K}$ and obtain $\lambda_1^\alpha\leq 4\sqrt{2K}h_1^\alpha$.
\end{proof}

The above proof of Theorem \ref{thm:Buser} is an extension of the proof of Ledoux
\cite[Theorem 5.2]{ledoux:04}, and the techniques to further establish the higher order estimate in Theorem \ref{thm:florentin} is a continuous analogue of the methods developed for discrete graphs in \cite[Theorem 5.1]{lmp:15}.
We also like to mention that higher order Buser type inequalities for the classical Laplacian $\Delta$ were first established by Funano \cite{funano:13}, and later improved in \cite{liu:14}.

This paper is structured as follows. In the next section, we introduce the general setting and review some fundamental facts. In Sections \ref{section:commutator} and \ref{section:Bochner}, we establish a Bochner type formula which is the foundation for our eigenvalue estimates. Finally, Theorem \ref{thm:Lichnerowicz} is proved in Section \ref{section:Lichnerowicz} and Theorem \ref{thm:florentin} is proved in Section \ref{section:Buser}.

\section{Preliminaries}\label{section:preliminaries}

\subsection{Magnetic Laplacian and its spectrum}

Let $(M,g)$ be a closed Riemannian manifold of dimension $n$ with Riemannian metric $g$. By abuse of notation, we denote by $\langle \cdot, \cdot \rangle$ the inner product induced by $g$ on the tangent bundle $TM$ or on the cotangent bundle $T^\ast M$. We extend the inner product $\langle \cdot, \cdot \rangle$ as a Hermitian inner product on the complexified tangent bundle $TM\otimes \C$ or on the complexified cotangent bundle $T^\ast M\otimes \C$. We will still use the same notation $\langle \cdot, \cdot \rangle$.

Let $\alpha$ be a smooth real differential $1$-form on $M$. Given a function $f\in C^\infty(M,\C)$, the operator $d^{\alpha}$
\begin{equation}\label{eq:alpha-derivative}
d^{\alpha}f:=df +if\alpha,
\end{equation}
maps $f$ to a smooth complex valued $1$-form on $M$. The magnetic Laplacian $\Delta^\alpha$ is defined as
\begin{equation}\label{eq:magnetic_Laplacian1}
\Delta^\alpha:=(d^\alpha)^\ast d^\alpha,
\end{equation}
where $(d^\alpha)^\ast$ is the formal adjoint of $d^\alpha$ with respect to the $L^2$-inner product of functions and $1$-forms. That is, for any $f\in C^\infty(M,\C)$ and any smooth complex valued $1$-form $\eta$, we have
\begin{equation}\label{eq:dalpha_dalpha_star}
\int_M\langle d^\alpha f, \eta\rangle\dvol=\int_Mf\overline{(d^\alpha)^\ast\eta}\dvol.
\end{equation}

There is a natural one-to-one correspondence from $T^\ast M\otimes \C$ to $T M\otimes \C$  via the musical isomorphism $\sharp:T^\ast M\otimes \C\rightarrow TM\otimes \C$ such that
	\begin{equation}\label{eq:musical}
	w(X)=\langle X,\conj{w^{\sharp}}\rangle,\qquad\forall\;X\in TM\otimes \C,\;\;\forall\;w\in T^\ast M\otimes \C.
	\end{equation}

Using the musical isomorphism $\sharp$, we have the following natural definitions.
\begin{defin}
Let $f\in C^\infty(M,\C)$. We define the \emph{magnetic gradient} $\gradA f$ of $f$ as
	\begin{equation}\label{eq:magnetic-gradient}
	\gradA f:=(d^\alpha f)^\sharp=\grad f+if\alpha^\sharp.
	\end{equation}
	
We define the \emph{magnetic divergence} $\dvgA X$ of a vector field $X\in TM\otimes \C$ as
	\begin{equation}\label{eq:magnetic-divergence}
	\dvgA X:=\dvg X +i\langle X,\alpha^\sharp\rangle.
	\end{equation}
\end{defin}

It is straightforward to check that $\dvgA$ is the formal adjoint operator of $\gradA$. In fact,
	\begin{equation}\label{eq:gradient_divergence}
	\begin{split}
	\int_M\langle \gradA f,X\rangle\dvol& =\int_M\langle\grad f,X\rangle+\langle if\alpha^\sharp,X\rangle\dvol\\
	&=-\int_M f\conj{\dvg X}\dvol+\int_M if\conj{\langle X,\alpha^\sharp\rangle}\dvol\\
	&=-\int_M f\conj{\dvg X+i\langle X,\alpha^\sharp\rangle}\dvol.
	\end{split}
	\end{equation}

\begin{prop}\label{prop:magnetic_laplacian}
For all $f\in C^\infty(M,\C)$, we have
\begin{equation}\label{eq:magentic_laplacian}
	\Delta^\alpha f=-\dvgA\gradA f=\Delta f-2i\langle \grad f,\alpha^\sharp\rangle+f(-i\dvg\alpha^\sharp+\abssqr{\alpha^\sharp}),
	\end{equation}
where $\Delta:=-\dvg\grad$ is the Laplace-Beltrami operator, and $\abssqr{\alpha^\sharp}:=\langle \alpha^\sharp, \alpha^\sharp\rangle$.
\end{prop}
\begin{proof}
By (\ref{eq:dalpha_dalpha_star}) and (\ref{eq:gradient_divergence}), we have
for any smooth complex valued $1$-form $\eta$, $$(d^\alpha)^\ast\eta=-\dvgA \eta^\sharp.$$ Recalling (\ref{eq:magnetic_Laplacian1}), this leads to
$$\Delta^\alpha f=-\dvgA (d^\alpha f)^\sharp=-\dvgA\gradA f.$$
Expanding the above formula, we obtain
	\begin{equation}
	\begin{split}
	\Delta^{\alpha} f
	&=-\dvgA\grad f -i\dvgA(f\alpha^\sharp)\\
	&=-\dvg\grad f -i\langle \grad f,\alpha^\sharp\rangle-i\dvg(f\alpha^\sharp) +f\abssqr{\alpha^\sharp}\\
	&=\Delta f-i\langle \grad f,\alpha^\sharp\rangle-if\dvg\alpha^\sharp -i\langle \grad f,\alpha^\sharp\rangle+f\abssqr{\alpha^\sharp}\\
	&=\Delta f-2i\langle \grad f,\alpha^\sharp\rangle+f(-i\dvg\alpha^\sharp+\abssqr{\alpha^\sharp}).
	\end{split}
	\end{equation}
\end{proof}

We now recall basic spectral properties of the magnetic Laplacian (see, e.g., \cite{shigekawa:87}, \cite{paternain:01}).
Let $L^2(M,\C)$ be the set of all complex valued square integrable functions with respect to the Riemannian volume measure. The densely defined operator $\Delta^\alpha$ on $L^2(M,\C)$ is essentially self-adjoint. In the sequel, we consider the self-adjoint extension of $\Delta^\alpha$ and still denote it by $\Delta^\alpha$. The operator $\Delta^\alpha$ has only discrete spectrum, and we list its eigenvalues with multiplicity as follows (\cite[Theorem 2.1]{shigekawa:87})
\begin{equation}\label{eq:magnetic_laplaican_eigenvalues}
0\leq \lambda_1^\alpha\leq \lambda_2^\alpha\leq \cdots \nearrow \infty.
\end{equation}
Similarly, we list the eigenvalues of $\Delta$ with multiplicity as
\begin{equation}\label{eq:laplaican_eigenvalues}
0= \lambda_1< \lambda_2\leq \cdots \nearrow \infty.
\end{equation}
As already mentioned in the introduction, the first eigenvalue $\lambda_1$ of $\Delta$ is zero and has multiplicity $1$. However, the first eigenvalue $\lambda_1^\alpha$ of $\Delta^\alpha$ can be positive and can have larger multiplicity. This can be seen explicitly in the following example, which was discussed in \cite[Example 1]{shigekawa:87}.

\begin{ex}\label{ex:S1one} Let $S_L^1=\R/L\Z$ be the circle of length $L$. We consider the $1$-form $\alpha:=Adx$ with $A\in \R$. Then, for any $f\in C^\infty(S_L^1,\C)$, we have
$$\Delta^\alpha f=-f''-2iAf'+A^2f,$$
and we have, for all $k\in \Z$,
$$\Delta^\alpha e^{i\frac{2\pi k}{L}x}=\left(\frac{2\pi k}{L}+A\right)^2e^{i\frac{2\pi k}{L}x}.$$
Since $\{e^{i\frac{2\pi k}{L}x}: k\in \Z\}$ is a Hilbert basis of $L^2(S_L^1,\C)$, the spectrum of $\Delta^\alpha$ is given by
$$\sigma(\Delta^\alpha)=\left\{\left(\frac{2\pi k}{L}+A\right)^2: k\in \Z\right\}.$$

In particular, we have $\lambda_1^\alpha>0$ for any choice of $A\not\in \{2\pi k/L: k\in \Z\}$. In the case of $A=-\pi/L$, we have $\lambda_1^\alpha=\pi^2/L^2$, whose eigenfunctions are $1$ and $e^{i\frac{2\pi}{L}x}$. Therefore, the first eigenvalue has multiplicity $2$.
\end{ex}

\subsection{Gauge transformation and Cheeger constants}\label{section:gauge_Cheeger}

In this subsection, we recall the Cheeger constants for magnetic Laplacians introduced in \cite{llpp:15}.

Let $U(1):=\{z\in \C: z\overline{z}=1\}$ and $C^\infty(M, U(1))$ be the set of smooth maps from $M$ to $U(1)$. A function $\tau\in C^\infty(M, U(1))$ can thus be viewed as a complex valued function on $M$, and we can define a smooth $1$-form $\alpha_\tau$ as follows:
\begin{equation}
\alpha_\tau:=\frac{d\tau}{i\tau}.
\end{equation}
Then every function $\tau\in C^\infty(M, U(1))$ gives rise to a \emph{gauge transformation}
\begin{equation}
\alpha\mapsto \alpha+\alpha_\tau,
\end{equation}
and the operators $\Delta^\alpha$ and $\Delta^{\alpha_\tau}$ are unitarily equivalent. In fact, we have (\cite[Proposition 3.2]{shigekawa:87})
\begin{equation}
\overline{\tau}\Delta^\alpha\tau=\Delta^{\alpha+\alpha_\tau}.
\end{equation}

Let $\mathcal{B}:=\{\alpha_\tau: \tau\in C^\infty(M,U(1))\}$, that is, $\mathcal{B}$ is the set of magnetic potentials which can be "gauged away". If $\alpha\in\mathcal{B}$, then $\Delta^\alpha$ is unitaritly equivalent to $\Delta$. The set $\mathcal{B}$ has the following characterization (\cite[Proposition 3.1 and Theorem 4.2]{shigekawa:87}):
\begin{thm}[Shigekawa] The following are equivalent:
\begin{itemize}
\item [(i)] $\alpha\in \mathcal{B}$;
\item [(ii)] $\lambda_1^\alpha=0$;
\item [(iii)] $d\alpha=0$ and $\int_C\alpha \equiv 0 \Mod{2\pi}$ for any closed curve $C$.
\end{itemize}
\end{thm}
Note that we have the following inclusions:
$$\{\text{exact $1$-forms}\}\subseteq \mathcal{B}\subseteq \{\text{closed $1$-forms}\}.$$

For any nonempty Borel subset $\Omega\subseteq M$, the \emph{frustration index} $\iota^\alpha(\Omega)$ of $\Omega$ is defined as (\cite[Definition 7.2]{llpp:15})
\begin{equation}
\iota^\alpha(\Omega)=\inf_{\tau\in C^\infty(\Omega, U(1))}\int_\Omega |(d+i\alpha)\tau|\dvol=\inf_{\eta\in \mathcal{B}_\Omega}\int_\Omega|\eta+\alpha|\dvol,
\end{equation}
where $\mathcal{B}_\Omega:=\{\alpha_\tau:\tau\in C^{\infty}(\Omega, U(1))\}$.
Note the frustration index $\iota^\alpha(M)$ measures, in some sense, the distance of $\alpha$ from the set $\mathcal{B}$.

\begin{defin}[Cheeger constant \cite{llpp:15}] Let $M$ be a closed Riemannian manifold with a smooth real differential $1$-form $\alpha$. For any Borel subset $\Omega$ of $M$, we denote
\begin{equation}
\phi^\alpha(\Omega):=\frac{\iota^\alpha(\Omega)+\mathrm{area}(\partial \Omega)}{\mathrm{vol}(\Omega)},
\end{equation}
where $\mathrm{vol}(\Omega)$ is the Riemannian volume of $\Omega$. The boundary measure $\mathrm{area}(\partial \Omega)$ is given by
\begin{equation}
\mathrm{area}(\partial \Omega):=\liminf_{r\to 0}\frac{\mathrm{vol}(\Omega_r)-\mathrm{vol}(\Omega)}{r},
\end{equation}
where $\Omega_r$ is the open $r$-neighbourhood of $\Omega$. Then the \emph{one-way (magnetic) Cheeger constant} $h_1^\alpha$ is defined as
\begin{equation}
h_1^\alpha:=\inf_{\Omega\subseteq M, \mathrm{vol}(\Omega)>0}\phi^\alpha(\Omega).
\end{equation}
Moreover, the \emph{$k$-way (magnetic) Cheeger constant} $h_k^\alpha$ is defined as
\begin{equation}
h_k^\alpha:=\inf_{\{\Omega_p\}_{[k]}}\max_{p\in [n]}\phi^\alpha(\Omega_p),
\end{equation}
where the infimum is taken over all possible $k$ disjoint subsets $\{\Omega_p\}_{[k]}$ with $\mathrm{vol}(\Omega_p)>0$ for every $p\in [k]:=\{1,2,\ldots,k\}$.
\end{defin}

All magnetic Cheeger constants are invariant under gauge transformation of the potential $\alpha$. In particular, when $\alpha \in \mathcal{B}$, $h_2^\alpha$ reduces to the classical Cheeger constant.

The following (higher order) Cheeger type inequalities were proved in \cite[Theorems 7.4 and 7.7]{llpp:15}.
\begin{thm}[\cite{llpp:15}]
Let $\alpha$ be a smooth real differential $1$-form on a closed connected Riemannian manifold $M$. Then we have
\begin{equation}
h_1^\alpha\leq 2\sqrt{2\lambda_1^\alpha}.
\end{equation}

Moreover, there exists an absolute dimension-independent constant $C>0$, such that for any closed connected Riemannian manifold $M$ with $\alpha$ and $k\in \N$, we have
\begin{equation}
h_k^\alpha\leq Ck^3\sqrt{\lambda_k^\alpha}.
\end{equation}
\end{thm}

\begin{ex}[$S_L^1$ revisited]\label{ex:S1two} Consider the circle $S_L^1$ with the real differential $1$-form $\alpha=Adx$, where $A\in \R$. The set $\mathcal{B}$ of magnetic potentials on $S_L^1$ which can be gauged away is given by
$$\mathcal{B}=\{f(x)dx: f\in C^\infty([0,L],\R), f(0)=f(L), \int_0^Lf(x)dx\equiv 0 \Mod{2\pi}\}.$$
We show now that the frustration index $\iota^\alpha(S_L^1)$ of $S_L^1$ is
\begin{equation}\label{eq:frustration_S}
\iota^\alpha(S_L^1)=\min_{k\in \Z}\left|2k\pi-AL\right|.
\end{equation}
Let $k_0\in \Z$ be the integer attaining the minimum of the expression at the right hand side of (\ref{eq:frustration_S}), and set $A_0:=|2k_0\pi-AL|/L$. Note $A_0\in [0,\pi/L]$.
Then we have $\iota^\alpha(S_L^1)\leq A_0L$ since $\frac{2k\pi}{L} dx\in \mathcal{B}$.

Suppose now that $\iota^\alpha(S_L^1)<A_0L$. Then there exists $f\in C^\infty([0,L],\R)$ satisfying
\begin{equation}\label{eq:fucntion_periodic}
f(0)=f(L)\qquad \text{ and }\qquad \int_0^Lf(x)dx\equiv 0 \Mod{2\pi},
\end{equation}
such that
\begin{equation*}
\int_0^L|f(x)-A|dx<A_0L.
\end{equation*}
This implies that there exists $f_0\in C^\infty([0,L],\R)$ satisfying (\ref{eq:fucntion_periodic}) such that
\begin{equation}\label{eq:to_be_contradicted}
\int_0^L|f_0(x)-A_0|dx<A_0L.
\end{equation}
In fact, we can set $f_0:=f-A+A_0$ when $AL\geq 2k_0\pi$ and $f_0=-f+A-A_0$ when $AL< 2k_0\pi$.
Then we have, by the triangle inequality,
\begin{equation*}
A_0L>\int_0^L|f_0(x)-A_0|dx\geq \int_0^L|f_0(x)|dx-A_0L\geq \left|\int_0^Lf_0(x)dx\right|-A_0L.
\end{equation*}
Since $2A_0L\leq 2\pi$, we must have
\begin{equation}\label{eq:f_0}
\int_0^Lf_0(x)dx=0,
\end{equation}
by (\ref{eq:fucntion_periodic}).
Finally, (\ref{eq:f_0}) implies
$$\int_0^L|f_0(x)-A_0|dx \geq \left|\int_0^L(f_0(x)-A_0)dx\right|= A_0L.$$
which contradicts to (\ref{eq:to_be_contradicted}). This proves (\ref{eq:frustration_S}).

On the other hand, for any proper subinterval $\Omega\subset S_L^1$, we have $\iota^\alpha(\Omega)=0$. Therefore, by definition, the one-way magnetic Cheeger constant of $S_L^1$ is
\begin{equation}\label{eq:Cheeger_S}
h_1^\alpha=\min\left\{\frac{2}{L}, \min_{k\in \Z}\left|\frac{2k\pi}{L}-A\right|\right\}.
\end{equation}
\end{ex}

\section{Commutator formulae}\label{section:commutator}

In this section, we derive the commutator formulae for the second order magnetic covariant derivative and the magnetic Hessian (see Definitions \ref{defin:magnetic_second_covder} and \ref{defin:magnetic_hessian} below). They are particularly useful in the next section for the derivation of a Bochner type formula.

Since the divergence is the trace of the Levi-Civita connection on $M$, we have
	\begin{equation}
	\begin{split}
	\dvgA X&=\dvg X+i\langle X,\alpha^\sharp\rangle=\sum_{j=1}^n\langle\nabla_{e_j} X,e_j\rangle+
	i\langle \sum_{j=1}^n\langle X,e_j\rangle e_j,\alpha^\sharp\rangle\\
	&=\sum_{j=1}^n\langle\nabla_{e_j} X,e_j\rangle+i\sum_{j=1}^n\alpha(e_j)\langle X,e_j\rangle
	=\sum_{j=1}^n\langle\nabla_{e_j} X+i\alpha(e_j) X,e_j\rangle.
	\end{split}
	\end{equation}
This suggests the following definition of magnetic covariant derivative.
	
\begin{defin}\label{defin:magnetic_covder}
Let $X,Y\in C^\infty(T M\otimes \C)$. We define the \emph{magnetic covariant derivative} of $X$ with respect to $Y$ as
	\begin{equation}\label{eq:magnetic-covder}
	\nablaA Y X:=\nabla_Y X+i\alpha(Y) X.
	\end{equation}
\end{defin}

Note that both the magnetic divergence and the magnetic covariant derivative are complex linear operators in all entries.

A direct calculation using Definition \ref{defin:magnetic_covder} leads to the following lemma.

\begin{lemma}\label{lemma:properties_covder}
For all $X,Y,Z\in C^\infty(T M\otimes \C)$ and for all $f\in C^\infty(M,\C)$, we have the following properties:

(i) (Riemannian property)
	\begin{equation}\label{eq:riemm_property}
	Z(\langle X,Y\rangle)=\langle\nablaA Z X ,Y\rangle+\langle X,\nablaA{\conj{Z}}Y\rangle;
	\end{equation}

(ii) (Leibniz rule)
	\begin{equation}\label{eq:leibniz_rule}
	\nablaA Y {(fX)}=Y(f)X+f\nablaA Y X;
	\end{equation}

(iii)\begin{equation}\label{eq:tensor-free}
	\nablaA{X}{Y}-\nablaA{Y}{X}=[X,Y]+i(\alpha(X)Y-\alpha(Y)X).
	\end{equation}
\end{lemma}

Similarly to the classical case, we define the second order magnetic covariant derivative.

\begin{defin}\label{defin:magnetic_second_covder}
For all vector fields $X,Y,Z\in C^\infty(TM\otimes \C)$, the \emph{second order magnetic covariant derivative} is the operator
	\begin{equation}\label{eq:magnetic_second_covder}
	\nablaA{X,Y}{Z}:=\nablaA X {(\nablaA Y Z)}-\nablaA{\nablaA{X}{Y}}{Z}.
	\end{equation}
\end{defin}

We now present a commutator formula that links the Riemannian curvature tensor with the second order magnetic covariant derivative.

For vector fields $U,V,W\in C^\infty(T M\otimes \C)$, we extend the Riemannian curvature tensor as
	\begin{equation}\label{eq:riem_curvature_tensor}
	R(U,V)W=\nabla_U\nabla_V \conj W-\nabla_V\nabla_U\conj W-\nabla_{[U,V]}\conj W,
	\end{equation}
 such that $R$ is complex linear in the first and second entry and complex anti-linear in the third entry. This implies that, for any $U,V\in C^\infty(T M\otimes\C)$,  $\langle R(U,V)V,U\rangle$ is real valued.

\begin{lemma}\label{lemma:commutator_riem_curvature}
For all $X,Y,Z\in C^\infty(TM\otimes \C)$, we have
	\begin{align}\label{eq:commutator_riem_curvature}
	\nabla^\alpha_{X,Y}Z-\nabla^\alpha_{Y,X}Z=R(X,Y)\conj Z&+id\alpha(X,Y)Z-i\nabla_{\alpha(X)Y-\alpha(Y)X}Z.
	\end{align}
\end{lemma}

\begin{proof}
Let us calculate explicitly the term $\nabla^\alpha_{X,Y}Z$ first. Applying (\ref{eq:magnetic_second_covder}), we check that
	\begin{align}
	\nabla^\alpha_{X,Y}Z=&\nabla^\alpha_X(\nabla^\alpha_YZ)-\nabla^\alpha_{\nabla^\alpha_XY}Z\notag\\
	=&\nabla^\alpha_X(\nabla_YZ)+i\nabla^\alpha_X(\alpha(Y)Z)-\nabla^\alpha_{\nabla_XY}Z-i\alpha(X)\nabla^\alpha_{Y}Z\notag\\
	=&\nabla_X(\nabla_YZ)+i\alpha(X)\nabla_YZ+i\nabla_X(\alpha(Y)Z)-\alpha(X)\alpha(Y)Z\notag\\
	&-\nabla_{\nabla_XY}Z-i\alpha(\nabla_XY)Z-i\alpha(X)\nabla_{Y}Z+\alpha(X)\alpha(Y)Z\notag\\
	=&\nabla_{X,Y}Z+i\alpha(Y)\nabla_XZ+iD\alpha(X;Y)Z, \label{eq:second_covder}
	\end{align}
where $$D\alpha(X;Y):=X(\alpha(Y))-\alpha(\nabla_XY).$$
Recall that we have (see, e.g., \cite[p. 366]{lee:09})
	\begin{align}\label{eq:dalpha}
	D\alpha(X; Y)-D\alpha( Y;X)
	&=X(\alpha(Y))-Y(\alpha(X))-\alpha([X, Y])\notag\\
	&=d\alpha(X,Y),
	\end{align}
and
	\begin{equation}\label{eq:comm_classical}
	\nabla^2_{X,Y}Z-\nabla^2_{Y,X}Z=R(X,Y)\conj Z.
	\end{equation}
Now (\ref{eq:second_covder}),  (\ref{eq:dalpha}), and (\ref{eq:comm_classical}) together imply this lemma.
\end{proof}
	
We proceed with one last definition.

\begin{defin}\label{defin:magnetic_hessian}
We define the \emph{magnetic Hessian} by
	\begin{equation}\label{eq:magntic_hessian}
	\hess^\alpha f(X,Y):=\langle\nabla^\alpha_X\grad^\alpha f,Y\rangle,
	\end{equation}
for all functions $f\in C^\infty(M,\C)$ and for all vector fields $X,Y\in C^\infty(T M\otimes \C)$.
\end{defin}

The magnetic Hessian is not Hermitian as in the classical case. In fact, we have the following commutator formula.

\begin{lemma}\label{lemma:commutator_hessian}
For all $X,Y\in C^\infty(T M\otimes \C)$ and $f\in C^\infty(M,\C)$, we have
	\begin{align}
	\hess^\alpha f(X,\conj{Y})-\hess^\alpha f(Y,\conj{X})
	= if d\alpha(X,Y).
	\end{align}
\end{lemma}

\begin{proof}
We first calculate $\langle\nabla^\alpha_X\grad^\alpha f,\conj{Y}\rangle$ explicitly. Using the Riemannian property (\ref{eq:riemm_property}),
	\begin{align}\label{eq:magnetic_Hessian_calculation}
	\langle\nabla^\alpha_X\grad^\alpha f,\conj{Y}\rangle=&X(\langle\grad^\alpha f,\conj Y\rangle)-\langle \grad^\alpha f, \nabla^\alpha_{\overline{X}}\conj Y\rangle\notag\\
	=&X(\langle \grad f+if\alpha^\sharp,\conj Y\rangle)-\langle \grad f+if\alpha^\sharp, \nabla_{\overline{X}}\conj Y+i\alpha(\overline{X})\conj Y\rangle\notag\notag\\
	=&X(\langle \grad f,\conj Y\rangle)-\langle\grad f, \nabla_{\overline{X}}\conj Y\rangle\notag\\
	&+iX(f\alpha(Y))+i\alpha(X) Y (f)-if\alpha(\nabla_X Y)-f\alpha(X)\alpha(Y)\notag\\
	=&\langle\nabla_X\grad f,\conj Y\rangle+iX(f)\alpha(Y)+iY(f)\alpha(X)-f\alpha(X)\alpha( Y)\notag\\
	&+ifD\alpha(X;Y).
	\end{align}
Observe that the entries in the first line of the last equality in (\ref{eq:magnetic_Hessian_calculation}) are symmetric w.r.t. $X$ and $Y$. In particular, we have $$\langle\nabla_X\grad f, \conj Y\rangle=\hess f(X,\conj Y)=\hess f(Y,\conj X).$$
Therefore, we conclude
	\begin{equation}\label{eq:Hessian_commutator_Dalpha}
	\hess^\alpha f(X,\conj Y)-\hess^\alpha f(Y,\conj X)=if\left(D\alpha(X; Y)-D\alpha(Y;X)\right).
	\end{equation}
Recalling (\ref{eq:dalpha}), we finish the proof.
\end{proof}

\section{A Bochner Type Formula for the Magnetic Laplacian}\label{section:Bochner}

We first recall that the Hilbert-Schmidt norm of the magnetic Hessian of a function $f \in C^\infty(M,\C)$ is
	\begin{equation}\label{eq:hessian_norm}
	\abssqr{\hessA f}=\sum_{i=1}^n\abssqr{\nablaA {e_i} {\gradA f}},
	\end{equation}
where $e_1,\ldots,e_n$ is an orthonormal real basis of $T_p M$. In fact,
	\begin{align*}
	\abssqr{\hessA f}&=\sum_{i,j=1}^n\abssqr{\hessA f (e_i,e_j)}=\sum_{i,j=1}^n\langle\nablaA{e_i}{\gradA f},e_j\rangle\langle e_j,\nablaA{e_i}{\gradA f}\rangle\notag\\
	&=\sum_{i=1}^{n}\langle \nablaA{e_i}{\gradA f}, \sum_{j=1}^n\langle \nablaA{e_i}{\gradA f},e_j\rangle e_j\rangle=\sum_{i=1}^n\langle \nablaA{e_i}{\gradA f},\nablaA{e_i}{\gradA f}\rangle.
	\end{align*}
	
\begin{thm}[Bochner type formula]\label{thm:magnetic_bochner}
Let $(M,g)$ be a complete Riemannian manifold of dimension $n$. Then, for all $f\in C^\infty(M,\C)$, we have
	\begin{multline}\label{eq:magnetic_bochner}
	-\frac{1}{2}\Delta(\abssqr{\gradA f})=\abssqr{\hessA f}-\frac{1}{2}\big(\langle \gradA f,\gradA(\Delta^\alpha f)\rangle+\langle \gradA(\Delta^\alpha f),\gradA f\rangle\big)\\
	+\ric(\gradA f,\gradA f) +i\big(d\alpha(\gradA f,\conj{\gradA f})-d\alpha(\conj{\gradA f},\gradA f)\big)\\
		+\frac{i}{2}\big(\langle\conj{f}\gradA f,(\delta d\alpha)^\sharp\rangle-\langle f\conj{\gradA f},(\delta d\alpha)^\sharp\rangle\big),
	\end{multline}
	where $\delta$ denotes the formal adjoint of the exterior derivative on $(M,g)$.
\end{thm}

\begin{proof}
Let $p\in M$ and consider a normal real basis $e_1, \ldots, e_n$ at $p$, i.e., $\abssqr{e_i}=1$ and $\nabla_{e_i} e_j=0$ for all $i,j=1,\ldots, n$.

Using the Riemannian property of $\nabla$ and $\nablaA{}{}$, we calculate
	\begin{align}\label{eq:start1}
	-\frac{1}{2}\Delta&(\abssqr{\gradA f})=\frac{1}{2}\tr \hess (\abssqr{\gradA f})
	=\frac{1}{2}\sum_{i=1}^{n}\Big(\hess \abssqr{\gradA f}\Big)(e_i,e_i)\notag\\
	&=\frac{1}{2}\sum_{i=1}^{n} \langle \nabla_{e_i}\grad(\abssqr{\gradA f}),e_i\rangle\notag\\
	&=\frac{1}{2}\sum_{i=1}^n e_i\Big(\langle \grad(\abssqr{\gradA f}),e_i\rangle\Big)
	-\langle \grad(\abssqr{\gradA f}),\underbrace{\nabla_{e_i}e_i}_{=0}\rangle\notag\\
	&=\frac{1}{2}\sum_{i=1}^n e_i\Big(\langle\nablaA{e_i}{\gradA f},\gradA f\rangle+\langle\gradA f,\nablaA{e_i}{\gradA f}\rangle\Big)\notag\\
	&=\frac{1}{2}\sum_{i=1}^n e_i\big(\hessA f(e_i,\gradA f)+\conj{\hessA f(e_i,\gradA f)}\big).
\end{align}

It is now sufficient to analyse the first summand $\frac{1}{2}\sum_{i=1}^n e_i(\hessA f(e_i,\gradA f))$, as the second one will directly give us its conjugate.

Using Lemma \ref{lemma:commutator_hessian}, the fact that $\nablaA{e_i}{e_i}=i\alpha(e_i)e_i$ (since $\nabla_{e_i}e_i=0$) and the Riemannian property \eqref{eq:riemm_property}, we have
	\begin{align}\label{eq:start2}
	\frac{1}{2}\sum_{i=1}^n e_i&(\hessA f(e_i,\gradA f))=\frac{1}{2}\sum_{i=1}^n e_i\Big(\hessA f(\conj{\gradA f},e_i)+ifd\alpha(e_i,\conj{\gradA f})\Big)\notag\\
	&=\frac{1}{2}\sum_{i=1}^n e_i\Big(\langle \nablaA{\conj{\gradA f}}{\gradA f},e_i\rangle +ifd\alpha(e_i,\conj{\gradA f})\Big)\notag\\
	&=\frac{1}{2}\sum_{i=1}^n\langle\nablaA{e_i}{\nablaA{\conj{\gradA f}}{\gradA f}},e_i\rangle + \frac{1}{2}\sum_{i=1}^n\langle\nablaA{\conj{\gradA f}}{\gradA f},\nablaA{e_i}{e_i}\rangle\notag\\
	&\hspace{2cm}+\frac{i}{2}\sum_{i=1}^n\big[e_i(f)d\alpha(e_i,\conj{\gradA f})+fe_i(d\alpha(e_i,\conj{\gradA f}))\big]\notag\\
	&=\frac{1}{2}\sum_{i=1}^n\langle\nablaA{e_i}{\nablaA{\conj{\gradA f}}{\gradA f}},e_i\rangle - \frac{i}{2}\langle\nablaA{\conj{\gradA f}}{\gradA f},\sum_{i=1}^n\alpha(e_i)e_i\rangle\notag\\
	&\hspace{2cm}+\frac{i}{2} d\alpha\left(\sum_{i=1}^ne_i(f)e_i,\conj{\gradA f}\right)+\frac{i}{2}\sum_{i=1}^n fe_i(d\alpha(e_i,\conj{\gradA f}))\notag\\
	&=\frac{1}{2}\sum_{i=1}^n\langle\nablaA{e_i}{\nablaA{\conj{\gradA f}}{\gradA f}},e_i\rangle - \frac{i}{2}\langle\nablaA{\conj{\gradA f}}{\gradA f},\alpha^\sharp\rangle\notag\\
	&\hspace{2cm}+\frac{i}{2} d\alpha(\grad f,\conj{\gradA f})+\frac{i}{2}\sum_{i=1}^n fe_i(d\alpha(e_i,\conj{\gradA f})).
	\end{align}


%
We analyse the first term of the RHS of (\ref{eq:start2}) using the definition of magnetic second covariant derivative, Lemma \ref{lemma:commutator_riem_curvature}, and equation \eqref{eq:tensor-free}.
	\begin{align}
	\frac{1}{2}&\sum_{i=1}^n\langle\nablaA{e_i}{\nablaA{\conj{\gradA f}}{\gradA f}},e_i\rangle=
	\frac{1}{2}\sum_{i=1}^n\langle\nablaA{e_i,\conj{\gradA f}}{\gradA f},e_i\rangle+\langle\nablaA{\nablaA{e_i}{\conj{\gradA f}}}{\gradA f},e_i\rangle\notag\\
	&=\frac{1}{2}\sum_{i=1}^n\Big[\langle\nablaA{\conj{\gradA f},e_i}{\gradA f},e_i\rangle+\langle R(e_i,\conj{\gradA f})\conj{\gradA f},e_i\rangle+i\langle d\alpha(e_i,\conj{\gradA f})\gradA f,e_i\rangle\notag\\
	&\hspace{2.6cm}+\langle\nabla_{i\alpha(\conj{\gradA f})e_i-i\alpha(e_i)\conj{\gradA f}}\gradA f,e_i\rangle+\langle\nablaA{\nablaA{e_i}{\conj{\gradA f}}}{\gradA f},e_i\rangle\Big]\notag\\
	&=\frac{1}{2}\sum_{i=1}^n\Big[\langle\nablaA{\conj{\gradA f}}{\nablaA{e_i}{\gradA f}},e_i\rangle+\langle R(e_i,\conj{\gradA f})\conj{\gradA f},e_i\rangle+i\langle d\alpha(e_i,\conj{\gradA f})\gradA f,e_i\rangle\notag\\
      &\hspace{1.5cm}+\langle\nablaA{\nablaA{e_i}{\conj{\gradA f}}-\nablaA{\conj{\gradA f}}{e_i}}{\gradA f}\rangle+\langle\nabla_{i\alpha(\conj{\gradA f})e_i-i\alpha(e_i)\conj{\gradA f}}\gradA f,e_i\rangle\Big]\notag
     \end{align}
      \begin{align}\label{eq:first}
	&=\frac{1}{2}\sum_{i=1}^n\Big[\langle\nablaA{\conj{\gradA f}}{\nablaA{e_i}{\gradA f}},e_i\rangle+\langle R(e_i,\conj{\gradA f})\conj{\gradA f},e_i\rangle\notag\\
	&\hspace{1cm}+i\langle d\alpha(e_i,\conj{\gradA f})\gradA f,e_i\rangle+\langle\nablaA{[e_i,\conj{\gradA f}]}{\gradA f},e_i\rangle\notag\\
	&\hspace{1cm}+\langle\underbrace{\nablaA{i(\alpha(e_i)\conj{\gradA f}-\alpha(\conj{\gradA f})e_i)}{\gradA f}+\nabla_{i(\alpha(\conj{\gradA f})e_i-\alpha(e_i)\conj{\gradA f})}\gradA f}_{=i\alpha(i\alpha(e_i)\conj{\gradA f}-i\alpha(\conj{\gradA f})e_i)\gradA f=0},e_i\rangle\Big]\notag\\
	&=\frac{1}{2}\ric(\gradA f,\gradA f)+\frac{1}{2}\sum_{i=1}^n\langle\nablaA{\conj{\gradA f}}{\nablaA{e_i}{\gradA f}},e_i\rangle\notag\\
	&\hspace{1cm}+\frac{1}{2}\sum_{i=1}^n\langle\nablaA{[e_i,\conj{\gradA f}]}{\gradA f},e_i\rangle+\frac{i}{2}d\alpha(\gradA f,\conj{\gradA f}).
	\end{align}

We now compute the second and third term of the RHS of (\ref{eq:first}). Using the Riemannian property of the magnetic covariant derivative and the fact that $\nablaA{\gradA f}{e_i}=i\alpha(\gradA f)e_i$ due to choice of the basis, we have
%
	\begin{align}\label{eq:(a)}
	\frac{1}{2}\sum_{i=1}^n\langle\nablaA{\conj{\gradA f}}{\nablaA{e_i}{\gradA f}}&,e_i\rangle=\frac{1}{2}\sum_{i=1}^n\conj{\gradA f}(\langle\nablaA{e_i}{\gradA f},e_i\rangle)-\langle\nablaA{e_i}{\gradA f},\nablaA{\gradA f}{e_i}\rangle\notag\\
	&=-\frac{1}{2}\conj{\gradA f}(\Delta^\alpha f)+\frac{i}{2}\alpha(\conj{\gradA f})\sum_{i=1}^n\langle\nablaA{e_i}{\gradA f},e_i\rangle\notag\\
	&=-\frac{1}{2}\langle\grad(\Delta^\alpha f),\gradA f\rangle-\frac{i}{2}\langle\conj{\gradA f},\alpha^\sharp\rangle\Delta^\alpha f\notag\\
	&=-\frac{1}{2}\langle\grad(\Delta^\alpha f)+i\Delta^\alpha f\alpha^\sharp, \gradA f\rangle\notag\\
	&=-\frac{1}{2}\langle\gradA (\Delta^\alpha f),\gradA f\rangle.
	\end{align}

Then, using the fact that $\nabla_{\conj{\gradA f}}e_i=0$ (due to the choice of the basis) and using Lemma \ref{lemma:commutator_hessian}, we have
%
	\begin{align}
	\frac{1}{2}\sum_{i=1}^n\langle\nablaA{[e_i,\conj{\gradA f}]}{\gradA f}&,e_i\rangle=\frac{1}{2}\sum_{i=1}^n\langle\nablaA{\nabla_{e_i}\conj{\gradA f}}{\gradA f},e_i\rangle=\frac{1}{2}\sum_{i=1}^n\hessA f(\nabla_{e_i}\conj{\gradA f},e_i)\notag\\
	&\hspace{-2cm}=\frac{1}{2}\sum_{i=1}^n\hessA f(e_i,\nabla_{e_i}\gradA f)+if d\alpha(\nabla_{e_i}\conj{\gradA f},e_i)\notag\\
	&\hspace{-2cm}=\frac{1}{2}\sum_{i=1}^n\langle\nablaA{e_i}{\gradA f},\nabla_{e_i}\gradA f\rangle+if d\alpha(\nabla_{e_i}\conj{\gradA f},e_i)\notag\\
	&\hspace{-2cm}=\frac{1}{2}\sum_{i=1}^n\langle\nablaA{e_i}{\gradA f},\nablaA{e_i}{\gradA f}\rangle+\langle\nablaA{e_i}{\gradA f},-i\alpha(e_i)\gradA f\rangle\notag\\
	&\hspace{-2cm}\hspace{5cm}+if d\alpha(\nabla_{e_i}\conj{\gradA f},e_i)\notag.
      \end{align}
Rewriting the first term on the RHS as a norm of a Hessian, we
obtain
       \begin{align}\label{eq:(b)}
       \frac{1}{2}\sum_{i=1}^n\langle\nablaA{[e_i,\conj{\gradA f}]}{\gradA f}&,e_i\rangle
	=\notag\\
       &\hspace{-2cm}=\frac{1}{2}\abssqr{\hessA f}+\frac{i}{2}\langle\nablaA{\sum_{i=1}^n\alpha(e_i)e_i}{\gradA f},\gradA f\rangle+\frac{i}{2}\sum_{i=1}^nf d\alpha(\nabla_{e_i}\conj{\gradA f},e_i)\notag\\
	&\hspace{-2cm}=\frac{1}{2}\abssqr{\hessA f}+\frac{i}{2}\langle\nablaA{\alpha^\sharp}{\gradA f},\gradA f\rangle+\frac{i}{2}\sum_{i=1}^nf d\alpha(\nabla_{e_i}\conj{\gradA f},e_i).
	\end{align}
	
Substituting equations \eqref{eq:(a)}, \eqref{eq:(b)} into \eqref{eq:first}, we obtain
	\begin{multline}\label{eq:first1}
	\frac{1}{2}\sum_{i=1}^n\langle\nablaA{e_i}{\nablaA{\conj{\gradA f}}{\gradA f}},e_i\rangle=
	\frac{1}{2}\ric(\gradA f,\gradA f)-\frac{1}{2}\langle\gradA (\Delta^\alpha f),\gradA f\rangle\\
	+\frac{1}{2}\abssqr{\hessA f}+\frac{i}{2}\langle\nablaA{\alpha^\sharp}{\gradA f}, \gradA f\rangle\\
	+\frac{i}{2}\sum_{i=1}^n fd\alpha(\nabla_{e_i}\conj{\gradA f},e_i)+\frac{i}{2}d\alpha(\gradA f,\conj{\gradA f}).
	\end{multline}

Consequently, substituting \eqref{eq:first1} into \eqref{eq:start2}, we obtain
	\begin{align}\label{eq:start3}
	\frac{1}{2}\sum_{i=1}^n e_i(\hessA f(&e_i,\gradA f))=\frac{1}{2}\ric(\gradA f,\gradA f)-\frac{1}{2}\langle\gradA (\Delta^\alpha f),\gradA f\rangle\notag\\
	&\hspace{2cm}+\frac{1}{2}\abssqr{\hessA f}+\frac{i}{2}d\alpha(\gradA f,\conj{\gradA f})\notag\\
	&\hspace{2cm}+\frac{i}{2}\sum_{i=1}^n f\big[d\alpha(\nabla_{e_i}\conj{\gradA f},e_i)+e_i(d\alpha(e_i,\conj{\gradA f})\big]\notag\\
	&\hspace{2cm}+\frac{i}{2}\langle\nablaA{\alpha^\sharp}{\gradA f}, \gradA f\rangle-\frac{i}{2}\langle\nablaA{\conj{\gradA f}}{\gradA f},\alpha^\sharp\rangle\notag\\
	&\hspace{2cm}+\frac{i}{2} d\alpha(\grad f,\conj{\gradA f}).
	\end{align}
	
We now combine the last three terms using Lemma \ref{lemma:commutator_hessian} and the definition of magnetic gradient.
	\begin{align}\label{eq:(d)}
	\frac{i}{2}\langle\nablaA{\alpha^\sharp}{\gradA f}, &\gradA f\rangle-\frac{i}{2}\alpha(\nablaA{\conj{\gradA f}}{\gradA f})+\frac{i}{2} d\alpha(\grad f,\conj{\gradA f})\notag\\
	&=\frac{i}{2}\big(\hessA f(\alpha^\sharp,\gradA f)-\hessA f(\conj{\gradA f},\alpha^\sharp)+ d\alpha(\grad f,\conj{\gradA f})\big)\notag\\
	&=\frac{i}{2}d\alpha(\grad f+if\alpha^\sharp,\conj{\gradA f})=\frac{i}{2}d\alpha(\gradA f,\conj{\gradA f}).
	\end{align}

	
Moreover,
the terms involving the sum give
	\begin{align}\label{eq:(e)}
	\frac{i}{2}\sum_{i=1}^nf\big[e_i(d\alpha&(e_i,\conj{\gradA f}))+d\alpha(\nabla_{e_i}\conj{\gradA f},e_i)\big]\notag\\
	&=\frac{i}{2}f\sum_{i=1}^n\big[e_i(d\alpha(e_i,\conj{\gradA f}))-d\alpha(e_i,\nabla_{e_i}\conj{\gradA f})-d\alpha(\nabla_{e_i}e_i,\conj{\gradA f})\big]\notag\\
	&=\frac{i}{2}f\sum_{i=1}^n(\nabla_{e_i}d\alpha)(e_i,\conj{\gradA f})\notag\\
	&=-\frac{i}{2}f\delta d\alpha(\conj{\gradA f}).
	\end{align}
For the last equality of (\ref{eq:(e)}), see, e.g., \cite[Def. 13.155 and Eq. (13.11)]{lee:09}.

Substituting \eqref{eq:(d)} and \eqref{eq:(e)} into \eqref{eq:start3} we have
	\begin{align}
	\frac{1}{2}\sum_{i=1}^n e_i(\hessA f(e_i,&\gradA f))=\frac{1}{2}\ric(\gradA f,\gradA f)-\frac{1}{2}\langle\gradA (\Delta^\alpha f),\gradA f\rangle\notag\\
	&+\frac{1}{2}\abssqr{\hessA f}+id\alpha(\gradA f,\conj{\gradA f})-\frac{i}{2}f\langle\conj{\gradA f},(\delta d\alpha)^\sharp\rangle.
	\end{align}

Finally, summing the above with its conjugate and substituting into \eqref{eq:start1}, we conclude
	\begin{align}
	-\frac{1}{2}\Delta(\abssqr{\gradA f})
	&=-\frac{1}{2}\big(\langle\gradA (\Delta^\alpha f),\gradA f\rangle+\langle\gradA f,\gradA (\Delta^\sharp f)\big)\notag\\
	&+\ric(\gradA f,\gradA f)+\abssqr{\hessA f}\notag\\
	&+i\big(d\alpha(\gradA f,\conj{\gradA f})-d\alpha(\conj{\gradA f},\gradA f)\big)\notag\\
	&+\frac{i}{2}\big(\langle\conj{f}\gradA f,(\delta d\alpha)^\sharp\rangle-\langle f\conj{\gradA f},(\delta d\alpha)^\sharp\rangle\big).
	\end{align}
\end{proof}

We now derive an integrated version of the Bochner type formula.


\begin{cor}\label{cor:integral_magnetic_bochner}
Let $M$ be a closed Riemannian manifold of dimension $n$. Then, for all $f\in C^\infty(M,\C)$ we have
	\begin{multline}\label{eq:integral_magnetic_bochner}
	\int_M\abssqr{\hessA f} \dvol+\int_M\ric(\gradA f,\gradA f) \dvol\\
	-\int_M\Re(\langle\gradA (\Delta^\alpha f),\gradA f\rangle) \dvol
	-\int_M\abssqr{f}\abssqr{d\alpha} \dvol\\
	+\int_M\Re(id\alpha(\gradA f,\conj{\gradA f})) \dvol=0,
	\end{multline}
	where $\Re(\cdot)$ stands for the real part of the corresponding complex number.
\end{cor}

\begin{proof}
Since $M$ is closed, the LHS of the Bochner Formula \eqref{eq:magnetic_bochner} is zero under integration.
%
Furthermore, we calculate
	\begin{align}
	\int_M \frac{i}{2}\langle \conj{f}\gradA f,(\delta d\alpha)^\sharp\rangle \dvol
	=\int_M\frac{i}{2}\langle\conj{f} d^\alpha f,\delta d\alpha\rangle \dvol
	=\int_M\frac{i}{2}\langle d(\conj{f} d^\alpha f),d\alpha\rangle \dvol,\notag
	\end{align}
where
	\begin{align}
	d(\conj{f} d^\alpha f)&=d\conj{f}\wedge d^\alpha f+\conj{f}d(d^\alpha f)\notag
	\\
	&=d\conj{f}\wedge df+if d\conj{f}\wedge\alpha+i\conj{f}df\wedge\alpha+i\abssqr{f}d\alpha\notag\\
	&=\conj{d^\alpha f}\wedge d^\alpha f+i\abssqr{f}d\alpha.\notag
	\end{align}
That is, we have
	\begin{align}\label{eq:2}
	\int_M \frac{i}{2}\langle \conj{f}\gradA f,&(\delta d\alpha)^\sharp\rangle \dvol=\int_M\frac{i}{2}\langle\conj{d^\alpha f}\wedge d^\alpha f,d\alpha\rangle \dvol-\int_M\frac{1}{2}\abssqr{f}\abssqr{d\alpha} \dvol\notag\\
	&=\int_M\frac{i}{2}d\alpha(\conj{\gradA f},\gradA f) \dvol-\frac{1}{2}\int_M\abssqr{f}\abssqr{d\alpha} \dvol,
	\end{align}
and similarly for its conjugate. Therefore, integrating formula \eqref{eq:magnetic_bochner}, we prove \eqref{eq:integral_magnetic_bochner}.
%
\end{proof}

\section{Lichnerowicz type estimates}\label{section:Lichnerowicz}
In this section, we prove Theorem \ref{thm:Lichnerowicz}, namely an upper bound for $\lambda_1^\alpha$ and a lower bound for $\lambda_2^\alpha$ and a spectral gap between them in the case of a positive lower Ricci curvature bound $K$ and small $\Vert d\alpha \Vert_\infty$.

\begin{proof}[Proof of Theorem \ref{thm:Lichnerowicz}]
Let $f$ be a normalized eigenfunction relative to $\lambda_1^\alpha$, i.e. $\Delta^\alpha f=\lambda^\alpha_1 f$. Then, $\int_M\abssqr{\gradA f} \dvol=\lambda^\alpha_1$, and Corollary \ref{cor:integral_magnetic_bochner} simplifies to
	\begin{multline}\label{eq:1.1}
	\int_M\abssqr{\hessA f} \dvol+\int_M\ric(\gradA f,\gradA f) \dvol- (\lambda^\alpha_1)^2\\
	-\int_M\abssqr{f}\abssqr{d\alpha} \dvol
	+\int_M\Re(id\alpha(\gradA f,\conj{\gradA f})) \dvol=0.
	\end{multline}

We now bound all the terms from below.
For an orthonormal basis $e_1,\ldots,e_n$, we have, using the Cauchy-Schwartz inequality,
	\begin{align}
	\abssqr{\hessA f}&=\sum_{i=1}^n\abssqr{\nablaA{e_i}{\gradA f}}\geq\sum_{i=1}^n\abssqr{\langle\nablaA{e_i}{\gradA f},e_i\rangle}\notag\\
	&\geq\frac{1}{n} \abssqr{\sum_{i=1}^n\langle\nablaA{e_i}{\gradA f},e_i\rangle}=\frac{1}{n}\abssqr{\Delta^\alpha f},\notag
	\end{align}
and therefore
	\begin{equation*}
	\int_M\abssqr{\hessA f} \dvol\geq\frac{1}{n}(\lambda^\alpha_1)^2.
	\end{equation*}
The curvature condition gives
	\begin{equation*}
	\int_M\ric(\gradA f,\gradA f) \dvol\geq K\int_M\abssqr{\gradA f} \dvol=K\lambda^\alpha_1.
	\end{equation*}
Moreover,
	\begin{equation*}
	-\int_M\abssqr{f}\abssqr{d\alpha} \dvol\geq -\normsqr{d\alpha}_\infty,
	\end{equation*}
and
	\begin{equation*}
	\int_M\Re(id\alpha(\gradA f,\conj{\gradA f})) \dvol\geq -\int_M\Vert d\alpha\Vert_\infty\abssqr{\gradA f} \dvol=-\lambda^\alpha_1\Vert d\alpha\Vert_\infty.
	\end{equation*}
Substituting all of the above into \eqref{eq:1.1}, we obtain
	\begin{equation}\label{eq:2.2}
	\big(1-\frac{1}{n}\big)(\lambda^\alpha_1)^2-(K-\lVert d\alpha\rVert_\infty)\lambda^\alpha_1+\normsqr{d\alpha}_\infty\geq 0.
	\end{equation}
	
We now consider the magnetic field $\epsilon\alpha$ with $\epsilon \in [0,1]$. Then, the eigenvalues $\lambda_j^{\epsilon \alpha}$ of the magnetic Laplacian depend continuously on $\epsilon$, and the above inequality becomes
	\begin{equation}
	\big(1-\frac{1}{n}\big)(\lambda^{\epsilon\alpha}_1)^2-(K-\lVert \epsilon d\alpha\rVert_\infty)\lambda^{\epsilon \alpha}_1+\normsqr{\epsilon d\alpha}_\infty\geq 0.
	\end{equation}

 When $\epsilon=\epsilon_0=0$, i.e., in absence of magnetic potential, the above inequality reduces to the classical Lichnerowicz Theorem giving the solutions $\lambda_1=0$ and $\lambda_2\geq\frac{nK}{n-1}$.  As $\epsilon$ starts to increase from $\epsilon_0=0$, $\lambda^{\epsilon\alpha}_1$ and $\lambda^{\epsilon\alpha}_2$ vary continuously but are still separated by an interval of positive length, as long as $\Vert d\alpha\Vert_\infty\leq \left(2\sqrt{\frac{n-1}{n}}+1\right)^{-1}K$. Namely, inequality \eqref{eq:2.2} gives the solutions
	\begin{equation}
	0<\lambda^{\epsilon\alpha}_1\leq \frac{n(K-\lVert \epsilon d\alpha\rVert_\infty)-n\sqrt{(K-\lVert d\alpha\rVert_\infty)^2-4(\frac{n-1}{n})\normsqr{d \alpha}_\infty}}{2(n-1)},
	\end{equation}
	\begin{equation}
	\lambda^{\epsilon\alpha}_2\geq\frac{n(K-\lVert \epsilon d\alpha\rVert_\infty)+n\sqrt{(K-\lVert d\alpha\rVert_\infty)^2-4(\frac{n-1}{n})\normsqr{d \alpha}_\infty}}{2(n-1)}.
	\end{equation}

Consequently, we obtain the spectral gap
	\begin{equation}
	\lambda^{\epsilon\alpha}_2-\lambda^{\epsilon\alpha}_1\geq\frac{n\sqrt{(K-\lVert d\alpha\rVert_\infty)^2-4(\frac{n-1}{n})\normsqr{d \alpha}_\infty}}{n-1}.
	\end{equation}
\end{proof}

\section{Buser type estimates}\label{section:Buser}
In this section, we prove Theorem \ref{thm:florentin}, namely the estimate
$$ 2\sqrt{t} \cdot h_k^\alpha \ge \frac{1}{k} - e^{-t \lambda_k^\alpha} $$
for all $k \in \N$ and $t \in [0,1/2K]$ in the case of $d\alpha = 0$ and a non-positive lower
Ricci curvature bound $-K$. Before we start with the proof, we recall the following example.

\begin{ex}[$S_L^1$ revisited]\label{ex:S1three} Consider the circle $S_L^1$ with the real differential $1$-form $\alpha=Adx$, where $A\in \R$. Recall from Examples \ref{ex:S1one} and \ref{ex:S1two} that
$$\lambda_1^\alpha=\min_{k\in \Z}\left(\frac{2\pi k}{L}+A\right)^2,\quad h_1^\alpha=\min\left\{\frac{2}{L},\min_{k\in \Z}\left|\frac{2\pi k}{L}-A\right|\right\}.$$
Therefore, in the case $\min_{k\in \Z}\left|\frac{2\pi k}{L}-A\right|\leq \frac{2}{L}$ we have $\lambda_1^\alpha=(h_1^\alpha)^2$.
\end{ex}

\begin{rmk}\label{rm:dal}
In the above example, we have $d\alpha=0$. Note that in the case $A\not\in \{2\pi k/L, k\in\mathbb{Z}\}$ we have $\alpha \not\in\mathcal{B}$, i.e., $\alpha$ cannot be gauged away.
\end{rmk}

Now we present the proof of Theorem \ref{thm:florentin}.
First note that, in the case $d\alpha=0$, the Bochner formula in Theorem \ref{thm:magnetic_bochner} reduces as follows.

\begin{lemma}\label{lemma:magnetic_bochner_dalphz0}
Let $(M,g)$ be a closed Riemannian manifold with a magnetic potential $\alpha$ such that $d\alpha=0$. Then, for all $f\in C^\infty(M,\C)$, we have
	\begin{multline}\label{eq:magnetic_bochner_buser}
	-\frac{1}{2}\Delta(\abssqr{\gradA f})=\abssqr{\hessA f}-\frac{1}{2}\big(\langle \gradA f,\gradA(\Delta^\alpha f)\rangle+\langle \gradA(\Delta^\alpha f),\gradA f\rangle\big)\\
	+\ric(\gradA f,\gradA f).
	\end{multline}
\end{lemma}

Let us denote by $(P_t^\alpha)_{t\geq 0}$ the heat semigroup corresponding to $\Delta^\alpha$. We write $(P_t)_{t\geq 0}$ for the classical heat semigroup.

\begin{lemma}\label{lemma:Buser}
Let $(M,g)$ be a complete Riemannian manifold with a magnetic potential $\alpha$ such that $d\alpha=0$. Let $-K$, $K\geq 0$ be a lower bound of the Ricci curvature of $M$. Then for any $f\in C^\infty(M, \C)$, we have the pointwise inequalities
\begin{itemize}
\item [(i)] $|\gradA (P_t^\alpha f)|^2\leq e^{2Kt}P_t(|\gradA f|^2),\,\, \forall\,\, t\geq 0$;
\item [(ii)] $P_t(|f|^2)-|P_t^\alpha f|^2\geq \frac{1-e^{-2Kt}}{K}|\gradA(P_t^\alpha f)|^2,\,\,\forall\,\, t\geq 0$, where $\left.\frac{1-e^{-2Kt}}{K}\right|_{K=0}:=2t$;
\item [(iii)] $\Vert f-P_t^\alpha f\Vert_1\leq 2\sqrt{t}\Vert\gradA f\Vert_1,\,\,\forall\,\,0\leq t\leq \frac{1}{2K}$.
\end{itemize}
\end{lemma}
\begin{rmk}
In fact, we will show $\ric\geq -K\Rightarrow (i) \Rightarrow (ii) \Rightarrow (iii)$.
\end{rmk}
\begin{proof}
Let $f$ be any smooth complex valued functions on $M$. For $0\leq s\leq t$, we define (at some point $x \in M$, which we suppress for the sake of readability)
$$F(s):=e^{2Ks}P_s(|\gradA P_{t-s}^\alpha f|^2).$$
Using the facts $\frac{\partial }{\partial s}P_s=-\Delta P_s=-P_s\Delta$ and $\frac{\partial}{\partial s}P_s^\sigma=-\Delta^\alpha P_s^\alpha=-P_s^\alpha\Delta^\alpha$, we calculate
\begin{align*}
\frac{d}{ds}F(s)=&2e^{2Ks}P_s\left(-K|\gradA(P_{t-s}^\alpha f)|-\frac{1}{2}\Delta(\abssqr{\gradA f})\right.\\
&\qquad\qquad\qquad\left.+\frac{1}{2}\big(\langle \gradA f,\gradA(\Delta^\alpha f)\rangle+\langle \gradA(\Delta^\alpha f),\gradA f\rangle\big)\right).
\end{align*}
Now applying Lemma \ref{lemma:magnetic_bochner_dalphz0} and the fact $P_s\geq 0$, we conclude $\frac{d}{ds}F(s)\geq 0$. Note further that $F(0)=|\gradA (P_{t-s}^\alpha f)|^2$ and $F(t)=e^{2Kt}P_t(|\gradA f|^2)$. This leads to $(i)$.

We then show $(i)\Rightarrow (ii)$. For $0\leq s\leq t$, let $G(s):=P_s(|P_{t-s}^\alpha f|^2)$. Thus, we have
$$P_t(|f|^2)-|P_t^{\alpha}f|^2=G(t)-G(0)=\int_0^tG'(s)ds,$$
where
\begin{align*}
G'(s)=&P_s\left(-\Delta(|P_{t-s}^\alpha f|^2)+\Delta^\alpha P_{t-s}^\alpha f\overline{P_{t-s}^\alpha f}+\overline{\Delta^\alpha P_{t-s}^\alpha f}P_{t-s}^\alpha f\right)\\
=& 2P_s(|\gradA (P_{t-s}^\alpha)|^2) \\
\geq & 2e^{-2Ks}|\gradA (P_s^\alpha P^\alpha_{t-s} f)|^2=2e^{-2Ks}|\gradA (P_t^\alpha f)|^2.
\end{align*}
In the above inequality, we used (i). In the case $K>0$, we arrive at
\begin{align*}
G(t)-G(0)\geq \int_0^t2e^{-2Ks}ds|\gradA (P_t^\alpha f)|^2=\frac{1-e^{-2Kt}}{K}|\gradA (P_t^\alpha f)|^2.
\end{align*}
Note that in the case $K=0$ we have $\int_0^t2e^{-2Ks}ds=2t$. This finishes the proof
of $(i)\Rightarrow (ii)$.

It remains to show $(ii) \Rightarrow (iii)$. We will present the argument for the case $K>0$. The case $K=0$ can be shown similarly. Assuming $(ii)$, we derive directly for $0\leq t\leq \frac{1}{2K}$
\begin{equation}\label{eq:infty_norm}
\Vert \sqrt{P_t(|f|^2)}\Vert_\infty\geq \sqrt{t} \Vert \gradA (P_t^\alpha f)\Vert_\infty.
\end{equation}
In the above, we used the inequality $1-e^{-x}\geq \frac{x}{2}$ for $0\leq x\leq 1$. For any $\phi\in C^\infty(M,\C)$ with $\Vert \phi\Vert_\infty\leq 1$, we calculate
\begin{align*}
\int_M(f-P_t^\alpha f)\phi\dvol=&-\int_M\phi\int_0^t\frac{\partial}{\partial s}P_s^\alpha f ds\dvol\\
=&\int_M\int_0^t(\Delta^\alpha(P_s^\alpha f))\phi ds\dvol\\
=&\int_0^t\int_M\Delta^\alpha f\cdot(P_s^\alpha \phi)\dvol ds\\
=&\int_0^t\int_M\langle \gradA f,\gradA (P_s^\alpha \phi)\rangle\dvol ds,
\end{align*}
where we used $\Delta^\alpha P_s^\alpha=P_s^\alpha\Delta^\alpha$ and the self-adjointness of $P_t^\alpha$. Continuing the calculation, we arrive at
\begin{align*}
\int_M(f-P_t^\alpha f)\phi\dvol\leq \Vert \gradA f\Vert_1\int_0^t\Vert \gradA(P_s^\alpha \phi)\Vert_\infty ds.
\end{align*}
For $0<t\leq \frac{1}{2K}$, we apply (\ref{eq:infty_norm}) to obtain
\begin{align}\label{eq:duality}
\int_M(f-P_t^\alpha f)\phi\dvol\leq 2\sqrt{t}\Vert \gradA f\Vert_1,
\end{align}
where we used $\sqrt{P_s(|\phi|^2)}\leq |\phi|\leq 1$. Applying (\ref{eq:duality}) to a sequence of smooth functions $\{\phi_k\}$ with $\Vert \phi_k\Vert_\infty\leq 1$, approximating in the $L^2(M,\C)$-norm the following function:
$$\phi_\infty=\left\{
                \begin{array}{ll}
                  \frac{\overline{f-P_t^\alpha f}}{|f-P_t^\alpha f|}, & \hbox{if $|f-P_t^\alpha f|\neq 0$;} \\
                  0, & \hbox{otherwise,}
                \end{array}
              \right.
$$
leads to the proof of $(iii)$.
\end{proof}


\begin{proof}[Proof of Theorem \ref{thm:florentin}]
Let $\{\Omega_p\}_{[k]}, \,k\in\N$ be any $k$ disjoint Borel subsets with $\mathrm{vol}(\Omega_p)>0$ for each $p\in [k]$.
 For each $\Omega_p\subseteq M$, let $\tau_p: M\to \C$ be the function given by
\begin{equation}
\tau_p(x)=\left\{
            \begin{array}{ll}
              \tau_p(x)\in U(1), & \hbox{$x\in \Omega$;} \\
              0, & \hbox{otherwise,}
            \end{array}
          \right.
\end{equation}
such that $\left.\tau_p\right|_{\Omega}$ is the minimizer in the definition of $\iota^\alpha(\Omega_p)$, i.e., $\int_{\Omega_p}|(d+i\alpha)(\left.\tau_p\right|_{\Omega_p})|\dvol=\iota^\alpha(\Omega_p)$. Applying Lemma \ref{lemma:Buser} $(iii)$ to smooth complex valued functions approximating $\tau_p$ yields, for $0\leq t\leq \frac{1}{2K}$,
\begin{align*}
 2\sqrt{t}\left(\iota^\alpha(\Omega_p)+\mathrm{area}(\partial \Omega_p)\right)\geq &\int_M|\tau_p-P_t^\alpha \tau_p|\dvol\\
\geq &\int_M|\tau_p-P_t^\alpha \tau_p|\cdot|\tau_p|\dvol\geq \int_M\Re\left(\tau_p\cdot\overline{\tau_p-P_t^\alpha \tau_p}\right)\\
=&\Vert\tau_p\Vert_2^2-\Vert P_{\frac{t}{2}}^\alpha\tau_p\Vert_2^2.
\end{align*}
We remark that the corresponding estimate for $\Delta$ in \cite[Theorem 5.2]{ledoux:04}, although leading to an improved constant, seems not to be applicable here. Let $\{\psi_\ell\}_{\ell=1}^\infty$ be the orthonormal eigenfunctions corresponding to $\{\lambda_\ell^\alpha\}_{\ell=1}^\infty$. By the spectral theorem, we have
$$\Vert P_{\frac{t}{2}}^\alpha\tau_p\Vert_2^2=\sum_{\ell=1}^\infty e^{-t\lambda_\ell^\alpha}|\langle \tau_p,\psi_\ell\rangle|^2.$$
Furthermore, observe that
$$\Vert \tau_p\Vert_2^2=\sum_{\ell=1}^\infty|\langle \tau_p,\psi_\ell\rangle|^2=\mathrm{vol}(\Omega_p).$$
Thus, we have, for $0\leq t\leq \frac{1}{2K}$,
\begin{equation*}
 2\sqrt{t}\left(\iota^\alpha(\Omega_p)+\mathrm{area}(\partial \Omega_p)\right)\geq \mathrm{vol}(\Omega_p)-\sum_{\ell=1}^\infty e^{-t\lambda_\ell^\alpha}|\langle \tau_p,\psi_\ell\rangle|^2.
\end{equation*}
Therefore, for given $k\in\N$, we have
\begin{align}
 2\sqrt{t}\phi^\alpha(\Omega_p)\geq &1-\sum_{\ell=1}^{k-1}\frac{|\langle \tau_p,\psi_\ell\rangle|^2}{\mathrm{vol}(\Omega_p)}-e^{t\lambda_k^\alpha}\sum_{\ell=k}^\infty\frac{|\langle \tau_p,\psi_\ell\rangle|^2}{\mathrm{vol}(\Omega_p)}\notag\\
\geq & 1-\sum_{\ell=1}^{k-1}\frac{|\langle \tau_p,\psi_\ell\rangle|^2}{\mathrm{vol}(\Omega_p)}-e^{t\lambda_k^\alpha}.\label{eq:florentin_1}
\end{align}
Observing that the functions $\frac{\tau_p}{\sqrt{\mathrm{vol}(\Omega_p)}}$, $p\in [k]$, are orthonormal in $L^2(M, \C)$, we obtain
\begin{align*}
\sum_{p=1}^{k}\frac{|\langle \tau_p,\psi_\ell\rangle|^2}{\mathrm{vol}(\Omega_p)}=\sum_{p=1}^{k}\left|\left\langle \frac{\tau_p}{\sqrt{\mathrm{vol}(\Omega_p)}},\psi_\ell\right\rangle\right|^2\leq \Vert \psi_\ell\Vert^2=1.
\end{align*}
Thus, we arrive at
\begin{align*}
 \frac{1}{k}\sum_{p=1}^k\sum_{\ell=1}^{k-1}\frac{|\langle \tau_p,\psi_\ell\rangle|^2}{\mathrm{vol}(\Omega_p)}\leq 1-\frac{1}{k}.
\end{align*}
This implies that there exists a $p_0\in [k]$ such that
\begin{align*}
 \sum_{\ell=1}^{k-1}\frac{|\langle \tau_{p_0},\psi_\ell\rangle|^2}{\mathrm{vol}(\Omega_{p_0})}\leq 1-\frac{1}{k}.
\end{align*}
Applying (\ref{eq:florentin_1}) to the set $\Omega_{p_0}$, we obtain
\begin{equation*}
2\sqrt{t}\max_{p\in [k]}\phi^\alpha(\Omega_{p})\geq 2\sqrt{t}\phi^\alpha(\Omega_{p_0})\geq \frac{1}{k}-e^{-t\lambda_k^\alpha}.
\end{equation*}
This completes the proof.
\end{proof}

We finish this section with the following consequence of Theorem \ref{thm:florentin}.

\begin{cor}\label{cor:higher_Buser}
Let $M$ be a closed Riemannian manifold whose Ricci curvature is bounded from below by $-K$, $K\geq 0$. Let $\alpha$ be a magnetic potential such that $d\alpha=0$. Then we have for any $k\in \N$,
\begin{equation}\label{eq:higher_Buser}
\lambda^\alpha_k\leq 2\log(2k)\max\left\{K, 2k^2(h_k^\alpha)^2\right\}.
\end{equation}
\end{cor}
\begin{proof}
 Applying Theorem \ref{thm:florentin}, we obtain
\begin{equation*}
  2\sqrt{t}\cdot h_k^\alpha\geq \frac{1}{k}-e^{-t\lambda_k^\alpha}.
\end{equation*}
If $\lambda_k^\alpha>0$ and $\lambda_k^\alpha\geq 2\log(2k)K$, we choose $t=\frac{\log(2k)}{\lambda_k^\alpha}\leq \frac{1}{2K}$ and obtain $\sqrt{\lambda_k^\alpha}\leq 2k\sqrt{\log(2k)} h_k^\alpha$.
Hence, we have (\ref{eq:higher_Buser}).
\end{proof}

\section*{Acknowledgement} SL and NP acknowledge the financial support of the EPSRC Grant EP/K016687/1 "Topology, Geometry and Laplacians of Simplicial Complexes".



\bigskip
\bigskip
\bigskip

\begin{flushleft}

{\bf AMS Subject Classification: 58J50, 53C21, 58J35}\\[2ex]

\bigskip

\small

Michela Egidi\\
Technische Universit\"at Chemnitz\\
Fakult\"at f\"ur Mathematik\\
Reichenhainer Stra{\ss}e 41\\
09126 Chemnitz\\
Germany\\
e-mail: \texttt{michela.egidi@mathematik.tu-chemnitz.de}\\[2ex]

\goodbreak

Shiping Liu\\
Durham University\\
Department of Mathematical Sciences\\
Science Laboratories\\
Sout Road\\
Durham, DH1 3LE\\
United Kingdom\\
e-mail: \texttt{shiping.liu@durham.ac.uk}\\[2ex]

Florentin M\"unch\\
Universit\"at Potsdam\\
Institut f\"ur Mathematik\\
Campus Golm, Haus 9\\
Karl-Liebknecht-Straße 24-25\\
14476 Potsdam\\
e-mail: \texttt{chmuench@uni-potsdam.de}\\[2ex]

Norbert Peyerimhoff\\
Durham University\\
Department of Mathematical Sciences\\
Science Laboratories\\
Sout Road\\
Durham, DH1 3LE\\
United Kingdom\\
e-mail: \texttt{norbert.peyerimhoff@durham.ac.uk}

\normalsize

\end{flushleft}

\end{document}